\documentclass[reqno]{amsart}

\usepackage{graphicx,subfigure}

\usepackage[latin1]{inputenc}
\usepackage[english]{babel}

\usepackage{amsmath,amsthm,amsfonts,latexsym,amssymb}
\usepackage[colorlinks]{hyperref}
\hypersetup{linkcolor=blue,citecolor=blue,filecolor=black,urlcolor=blue}
\usepackage{comment}

\usepackage{color,amsthm,amsfonts}
\definecolor{darkgreen}{rgb}{0,0.7,0.1}

{ \theoremstyle{plain}
\newtheorem{theorem}{Theorem}
\newtheorem{proposition}[theorem]{Proposition}
\newtheorem{lemma}[theorem]{Lemma}
\newtheorem{corollary}[theorem]{Corollary}
  \theoremstyle{remark}

  \theoremstyle{definition}

}

\begin{document}


\title[]{Asymptotics for a high-energy solution of a supercritical problem}

\author[F. Colasuonno]{Francesca Colasuonno}
\address{Francesca Colasuonno\newline\indent
Dipartimento di Matematica
\newline\indent
Alma Mater Studiorum Universit\`a di Bologna
\newline\indent
piazza di Porta San Donato 5, 40126 Bologna, Italy}
\email{francesca.colasuonno@unibo.it}

\author[B. Noris]{Benedetta Noris}

\address{Benedetta Noris \newline\indent
Dipartimento di Matematica\newline\indent
Politecnico di Milano\newline\indent
Piazza Leonardo da Vinci 32, 20133 Milano, Italy}
\email{benedetta.noris@polimi.it}


\begin{abstract} 
In this paper we deal with the equation
\[-\Delta_p u+|u|^{p-2}u=|u|^{q-2}u\]
for $1<p<2$ and $q>p$, under Neumann boundary conditions in the unit ball of $\mathbb R^N$.
We focus on the three positive, radial, and radially non-decreasing solutions, whose existence for $q$ large is proved in \cite{BFGM}. 
We detect the limit profile as $q\to\infty$ of the higher energy solution and show that, unlike the minimal energy one, it converges to the constant $1$.
The proof requires several tools borrowed from the theory of minimization problems and accurate a priori estimates of the solutions, which are of independent interest. 
\end{abstract}

\maketitle

\section{Introduction}
Let us consider the following Neumann problem
\begin{equation}\label{eq:Pq}
\begin{cases}
-\Delta_p u+u^{p-1}=u^{q-1}\quad&\mbox{in }B,\\
u>0\quad&\mbox{in }B,\\
\partial_\nu u=0\quad&\mbox{on }\partial B,
\end{cases}
\end{equation}
where $B$ is a ball of $\mathbb R^N$ ($N\ge1$), $\nu$ is the outer unit normal of $\partial B$ and $1<p<q$.  
In the case $p=2$, this system is the stationary version of the Keller-Segel model for chemotaxis and for this reason it has been well studied starting from the `80. Clearly it admits the spatially homogeneous solution $1$ for every value of the exponent $q$ and in any ball $B$; the interest is in finding spatially inhomogeneous steady states for the Keler-Segel system, i.e. non-constant solutions of \eqref{eq:Pq}. The existence of non-constant solutions is subject to hypoteses on the radius of the domain $B$: for sufficiently small radii the constant solution is the only solution of the problem, while for sufficiently large radii there exists a non-constant solution, see \cite{LNT} for the Sobolev-subcritical case, and \cite{LN} for the supercritical regime. In the critical case the situation is more involved and the validity of similar existence/non-existence results strongly depends on the dimension $N$, cf. \cite{AY91,AY97,BKP}. We recall in passing that for the first time in \cite{AY97} the quasilinear problem with $p\neq 2$ is studied proving a non-existence result in the critical case.

Recently it has been proved that, even restricting one's attention to the search for radial solutions, this problem presents multiplicity and the structure of the set of radial solutions can be very rich, depending on the values of the parameters into play. On the other hand, very little is still known in the non-radial setting; for $p = 2$ we refer to \cite{CM,dPPV} for the existence of solutions in non-radial domains, and to \cite{BCNW} for the existence of a non-radial solution of the Dirichlet problem in an annulus.

For the moment let us focus on radial solutions that present a local minimum at the origin; the complementary case is more difficult to deal with, as will be clarified in the sequel.
In the case $p=2$, it has been proved in \cite{BGT} via bifurcation techniques that, as $q$ increases, each time that $q$ crosses the value $2+\lambda_{i+1}^{\text{rad}}$, a new non-constant radial solution of \eqref{eq:Pq} appears, where $\lambda_{i+1}^{\text{rad}}$ is the $i$-th non-zero eigenvalue of $-\Delta$ with Neumann boundary conditions in the ball, with $i\geq1$. These solutions can be distinguished upon the number of oscillations around the constant solution $1$, meaning that, for every integer $i\geq1$ there exists a solution that intersects $i$-times the constant $1$.
In \cite{ABF-ESAIM} we have studied the problem for every $p>1$. By means of shooting methods, we have proved the existence of infinitely many radial solutions  in the range $2<p<q$. As in the case $p=2$, these solutions can be distinguished upon the number of oscillations around $1$. In the same paper we also showed that a new phenomenon appears when $1<p<2$: solutions with the same number of oscillations appear in couple, i.e. for every integer $i\geq1$ there exist two distinct radial solutions that cross $i$-times the constant $1$. We mention that the last existence result requires that the domain of the equation is a ball with sufficiently large radius.  We address the interested reader to \cite{BGT,ABF-ESAIM} for a detailed comparison among the various cases and for numerical simulations.
 Concerning radial oscillating solutions that have a local maximum at the origin, their existence has been proved in  \cite{ABF-PRE} in the Sobolev-subcritical setting for any $p>1$ and in \cite{BGT} for $p=2$ and values of $q$ close to the bifurcation parameter $2+\lambda_{i+1}^{\text{rad}}$. These solutions may present the phenomenon of explosion at the origin in the Sobolev-supercritical regime, see \cite{CF}.

Since the analysis in the present paper is independent of the radius of the ball, from now on we let the domain $B$ be the unit ball of $\mathbb R^N$. We shall focus here on radial solutions that, in addition to having a local minimum at the origin,  are also monotone in the radial variable $r$ for every $r\in(0,1)$.  This additional property allows to characterize them variationally, as we shall now illustrate.  Following \cite{SerraTilli2011}, we define the cone
\begin{equation}\label{cone}
\mathcal C :=\Big\{u\in W^{1,p}_{\mathrm{rad}}(B)\,:\, u\ge0,\,u(r_1)\le u(r_2) \mbox{ for all }0<r_1\le r_2\le1\Big\},
\end{equation}
where, with abuse of notation, $u(|x|):=u(x)$.
As proved in \cite[Proposition 2.6]{secchi2012increasing} (see also \cite{BNW,BF,SerraTilli2011}), $\mathcal C $ is embedded in $L^\infty(B)$, i.e. there exists a constant $C(N,p)>0$ such that
\begin{equation}\label{eq:CNp}
\|u\|_{L^\infty(B)}\le C(N,p)\|u\|_{W^{1,p}(B)}\quad\mbox{for all }
u\in\mathcal{C} .
\end{equation}
Even more, the solutions of the problem that belong to the cone are a priori bounded in the $C^1$- norm, see \cite[Lemma 2.2]{BFGM}. This allows to define a truncated nonlinearity $f_q$ (see \eqref{eq:phi_s0} ahead) that coincides with $u^{q-1}$ up to the a priori bound and is Sobolev-subcritical at infinity. As a consequence, all solutions of the corresponding modified problem \eqref{eq:f_tilde_q} belonging to the cone $\mathcal C$ also solve the original problem \eqref{eq:Pq}. In view of this fact, it is possible to associate to problem \eqref{eq:Pq} \emph{inside the cone $\mathcal C$} an energy functional $I_q$ and hence to define a variational structure. We shall recall in Section \ref{Sec:variational_setting} the variational setting introduced in the previously mentioned articles to find solutions of \eqref{eq:Pq} in $\mathcal C$.

In the present paper we restrict our attention to the case
\[
1< p < 2.
\]
In this range,  we proved in  \cite[Theorems 1.1 and 1.3]{BFGM} the existence of two distinct solutions $u_q$ and $v_q$ of \eqref{eq:Pq}, for every $q$ sufficiently large. More precisely, $u_q$ and $v_q$ are both radial, non-negative, radially non-decreasing and can be distinguished upon their energy: $I_q(u_q)<I_q(1)<I_q(v_q)$.
As both $u_q$ and $v_q$ present a minimum at the origin and intersect the constant solution $1$ exactly once, this is coherent with the multiplicity result \cite{ABF-ESAIM} described above, obtained with shooting techniques. 
Beyond the mere existence, the variational approach allows to describe the nature of the solutions as critical points of $I_q$, for $q$ sufficiently large. Indeed, $u_q$ is a global minimizer of $I_q$ restricted to the set $\mathcal N_q$, which is the intersection of the Nehari manifold with the cone $\mathcal C$, see \eqref{eq:N_q_def} ahead. Contrarily to what happens in the regime $p\geq2$, the constant solution $1$ is a local (but not global) minimizer of the energy functional restricted to the same Nehari-type set $\mathcal N_q$ (see \cite[Theorem 1.2]{BFGM}). The presence of two minimizers on $\mathcal N_q$ justifies the existence of a third critical point $v_q$ of $I_q$, which is of minimax-type on $\mathcal N_q$.

The minimax characterization just described also allows us in \cite[Theorem 1.1]{BFGM} to have an insight on the limit profile of the lower-energy solutions $u_q$ as $q\to\infty$.
More precisely, as $q\to\infty$,
\begin{equation}\label{eq:conv-uq}
u_q  \to G  \textrm{ in } W^{1,p}(B) \cap C^{0,\nu}(\bar B) \textrm{ for any }\nu\in(0,1),
\end{equation}
where $G$ is the unique positive solution of the following $(p-1)$-homogeneous equation coupled with Dirichlet 1 boundary conditions
\begin{equation}\label{eqforG}
\begin{cases}-\Delta_p G+G^{p-1}=0\quad&\mbox{in }B,\\
G=1&\mbox{on }\partial B.
\end{cases}
\end{equation}
This asymptotic analysis plays a fundamental role in \cite{BFGM} to show that the lower-energy solution is not constant for $q$ large. Indeed, the limit problem \eqref{eqforG} does not admit any constant solution, hence by the convergence \eqref{eq:conv-uq}, $u_q\not\equiv 1$ for $q$ large. 
Moreover, in the case $p=2$, an analogous asymptotic estimate allows to prove in \cite{BGNT} the non-degeneracy of $u_q$ and its uniqueness as global minimizer of $I_q$ on $\mathcal N_q$, which are the starting steps for the subsequent construction of oscillating radial solutions.

The importance of the asymptotic analysis stimulated us to investigate the limit behavior of the higher-energy solutions $v_q$: this is the aim of the present paper, whose main result is stated in the following theorem.

\begin{theorem}\label{thm:vqto1}
Let $p\in (1,2)$ be fixed and let $v_q\in\mathcal C$ be a solution of  \eqref{eq:Pq} having energy higher than the constant $1$, namely
\begin{equation}\label{eq:higher_energy}
\int_B \left(\frac{|\nabla v_q|^p}{p}+\frac{|v_q|^p}{p}-\frac{|v_q|^q}{q} \right)dx > |B|\left( \frac{1}{p}-\frac{1}{q}\right).
\end{equation}
Then, as $q\to\infty$, 
\[v _q\to 1\quad \mbox{in }W^{1,p}(B)\cap C^{0,\nu}(\bar B)\] 
for every $\nu\in (0,1)$. 
\end{theorem}

As already mentioned, the existence of a high-energy solution $v_q$ as in Theorem \ref{thm:vqto1}, for $q$ sufficiently large, is proved in \cite[Theorem 1.3]{BFGM}.
We shall recall in Section \ref{Sec:variational_setting} the precise variational characterization of $v_q$, although we do not need it explicitly for the proof of Theorem \ref{thm:vqto1}.

We highlight that the result proved in Theorem \ref{thm:vqto1} is consistent with the numerical simulations in \cite[Fig.2]{ABF-ESAIM} and \cite[Fig.1]{BFGM}.  It is worth mentioning that,  even though Theorem \ref{thm:vqto1} applies to any solution in $\mathcal{C}$ having energy higher than the energy of the constant $1$, according to our numerical simulations there exists at most one of such solutions, thus being the solution $v_q$ described in Section \ref{Sec:variational_setting}.

In order to highlight the main difficulties in the proof of Theorem \ref{thm:vqto1}, let us compare it with the proof of \eqref{eq:conv-uq} in \cite{BFGM} (see also \cite{GrossiNoris} for the case $p=2$).  As already mentioned, all solutions of \eqref{eq:Pq} in $\mathcal C$ are a priori bounded in the  $C^1$-norm, this allows, for both $u_q$ and $v_q$, to get the existence of a limit profile and some of its properties, i.e. monotonicity and value 1 at the boundary of the ball. Since problem \eqref{eq:Pq} has multiple solutions in $\mathcal C$, the difficulty is to distinguish the limit profile of the higher-energy solutions $v_q$ from that of $u_q$: the energy levels have to be taken into account.   

As for $u_q$, since it is a gobal minimizer of $I_q$ on $\mathcal N_q$, it is natural to look for a limit minimization problem. We show in \cite{BFGM} that the suitable limit problem is the following 
\[
\inf\left\{\frac{1}{p}\|u\|^p_{W^{1,p}(B)}\,:\,u\in W^{1,p}(B), \,u(1)=1\right\},\]
that has a unique solution $G$, which also solves \eqref{eqforG}. 
Indeed, the term $\frac{1}{q}\int_B u_q^q\, dx$ in the energy functional $I_q$ is infinitesimal as $q\to\infty$ thanks to the a priori bounds on $u_q$.
Notice that the boundary conditions have changed in the limit problem; this is reminiscent of the fact that the limit problem presents indeed a singularity on the boundary of $B$.
To clarify this point,  we refer to the asymptotic analysis for one-peak and multi-peak solutions to the semilinear problem, namely for $p=2$, performed in \cite{BGNT}. As shown therein, the limit configuration satisfies an equation involving a one-dimensional Dirac delta, see \cite[equation (2.5)]{BGNT}. When dealing with non-decreasing solutions, as in the present paper, we can interpret the singularity as lying on the boundary of the domain and thus affecting the boundary conditions.

Concerning the analysis of the asymptotics for $v_q$, as these functions are not minimizers over $\mathcal N_q$, a major difficulty is that it is not clear how to detect a variational problem in the limit. 
This fact prevents us from taking advantage, in the proof, of the properties of the limit problem, such as uniqueness or shape of the solutions, thus making the analysis more involved than in the case of $u_q$.
Let us remark that both $u_q$ and $v_q$ intersect the constant $1$ exactly once, at possibly different points, and that the intersection point may vary with $q$. For radii smaller than the intersection point, if the solutions are strictly smaller than 1 uniformly in $q$, the right hand side of the equation in \eqref{eq:Pq} would vanish as $q\to\infty$; contrarily, for radii larger than the intersection point, if the solutions are strictly larger than 1 uniformly in $q$, the right hand side of the equation in \eqref{eq:Pq} would explode as $q\to\infty$; finally, in a neighbourhood of the intersection point, the right hand side of the equation is an indeterminate form $1^\infty$ as $q\to\infty$. These heuristic considerations suggest that very little about the limit configuration can be inferred from the equation itself.
The proof of Theorem \ref{thm:vqto1} should instead pass through an analysis of the energy levels.

Here is how we proceed. In Proposition \ref{prop:conv-en}, we detect the limit of the energy levels of $v_q$ thanks to a significant refinement of the $W^{1,p}$-bounds. The key tool for improving such bounds is the equivalent definition \eqref{eq:def-sol-equiv} of weak solutions, which is given in terms of a distributional inequality and is borrowed from the field of minimization problems.  
We obtain that the limit function $v_\infty$ of $(v_q)$ has the same energy as the constant function $1$. 
Nonetheless, this is not sufficient to conclude, mainly due to the fact that our a priori estimates are not strong enough to provide the $W^{1,p}$-convergence. This differs from other asymptotic studies, compare for example our Lemma \ref{commonbound} with \cite[equation (3.1)]{BMP}. 
In order to overcome this difficulty and conclude the proof, we argue by contradiction and suppose that $v_\infty$ is strictly below 1 in a ball $B_{\bar R}$, with $\bar R<1$. Working in the absurd hypothesis has the twofold advantage of having a limit minimization problem of the form \eqref{eqforG} in the smaller ball $B_{\bar R}$ and of allowing us to strengthen the convergence up to $C^1$ in any compact subset of $B_{\bar R}$, see Lemma \ref{lem:conv-v'q} ahead. Both ingredients are crucial to conclude that, in the absurd setting, the energy of the limit function is strictly less than the limit of the energies of $v_q$, thus providing the desired contradiction.

The paper is organized as follows. In Section \ref{Sec:variational_setting} we introduce the variational setting and the equivalent definition of weak solution for problem \eqref{eq:Pq}. We further prove therein a priori estimates for non-decreasing radial solutions and a general weak convergence result, of independent interest, holding when varying the parameter $q$. Finally, in Section \ref{Sec:main} we prove Theorem \ref{thm:vqto1} using the technique explained above.

\section{Preliminary results}\label{Sec:variational_setting}
\subsection{Variational setting}
Let us describe the variational setting introduced in \cite{BNW,BF,BFGM} to find solutions of \eqref{eq:Pq} in $\mathcal C$. A first rough $L^\infty$-estimate on solutions of \eqref{eq:Pq} belonging to $\mathcal C$ is $\|u\|_{L^\infty(B)}\le 1+(p')^{1/p}=:K_\infty$, where $p'$ is the conjugate exponent of $p$ (cf. \cite[Lemma 2.2]{BFGM}). In \cite{BFGM} we introduce the modified nonlinearity
\begin{equation}\label{eq:phi_s0}
f_q(s):=\begin{cases}s^{q-1}\quad&\mbox{if }s\in[0,s_0],\\
s_0^{q-1}+\frac{q-1}{\ell-1}s_0^{q-\ell}(s^{\ell-1}-s_0^{\ell-1})&\mbox{if } s\in(s_0,\infty),\end{cases}
\end{equation}
with $s_0:=\max \left\{ 2+(p')^{1/p},  C(N,p)(1+|B|^{1/p}) \right\}$, $C(N,p)$ is the constant introduced in \eqref{eq:CNp}, $\ell\in (p,p^*)$, and $p^*$ the critical Sobolev exponent. In view of the $L^\infty$-estimate, being $s_0>K_\infty$, it holds that every solution of the modified problem
\begin{equation}\label{eq:f_tilde_q}
\begin{cases}
-\Delta_p u+ u^{p-1}={f}_q(u)\quad&\mbox{in }B,\\
u>0&\mbox{in }B,\\
\partial_\nu u=0&\mbox{on }\partial B,
\end{cases}
\end{equation}
belonging to $\mathcal C $ also solves the original problem \eqref{eq:Pq}. 
Hence, when looking for solutions in $\mathcal C$, it is possible to associate to equation \eqref{eq:Pq} an energy functional that is well defined in $W^{1,p}(B)$, namely
\begin{equation*}
I_q (u):= \int_B\left(\frac{|\nabla u|^p}{p}+\frac{|u|^p}{p}-F_q(u) \right)dx,
\end{equation*}
where $F_q(u):=\int_0^u f_q(s)ds$.  Although $I_q $ is not the standard energy functional associated to \eqref{eq:Pq}, it has the property that its critical points belonging to the cone $\mathcal C $ are weak solutions of \eqref{eq:Pq}. This allows to investigate the existence of solutions to \eqref{eq:Pq} via variational methods inside $\mathcal C $.

In the above mentioned papers, the following Nehari-type set inside $\mathcal C $ is defined
\begin{equation}\label{eq:N_q_def}
\mathcal N_q :=\left\{u\in\mathcal C  \setminus\{0\}\,: 
\int_B(|\nabla u|^p+|u|^p)dx=\int_B f_q(u)u\,dx\right\}.
\end{equation}
As already mentioned in the Introduction, it is proved therein the existence of a non-constant non-decreasing radial solution $u_q $ of \eqref{eq:Pq} that achieves the critical level
\begin{equation}\label{eq:c_q_p_def}
\inf_{u\in \mathcal N_q } I_q (u),
\end{equation}
provided that $q>p$ when $p>2$, $q>2+\lambda_2^{\text{rad}}$ for $p=2$ (recall that $\lambda_2^{\text{rad}}$ is the first non-zero eigenvalue of the Laplacian in the ball $B$ under Neumann boundary conditions), and $q$ sufficiently large for $1<p<2$. 

In the following let $1<p<2$. In this case, in \cite[Theorem 1.2]{BFGM}, we prove the following estimate for every $w\in \mathcal{N}_q $ with the property $\|w-1\|_{W^{1,p}(B)}\leq\delta$:
\begin{equation}\label{eq:ordine-energie}
I_q (w)-I_q (1)\ge M\|w-1\|_{W^{1,p}(B)}^p
\end{equation}
for $q$ sufficiently large and some constants $\delta \in (0,1)$ and $M>0$.
This means that $1$ is a local minimizer for $I_q\big|_{\mathcal N_q}$ and allows to prove, for $q$ sufficiently large, the existence of a second non-constant solution $v_q\in\mathcal C$ which is of  mountain pass type over $\mathcal N_q $, cf. \cite[Theorem 1.3]{BFGM}. In particular, $v_q$ achieves the energy level
\begin{equation}\label{eq:d_q}
\inf_{\gamma\in\Upsilon_q } \max_{(t,s)\in Q} I_q (\gamma(t,s)),
\end{equation}
where $\Upsilon_q :=\{\gamma \in C(Q;\mathcal C) \,:\, \gamma=\gamma_0 \text{ on } \partial Q \}$, $Q:=[R_1,R_2]\times [0,1] \subset \mathbb{R}^2$, $0<R_1 \ll 1$, $R_2 \gg1$, and $\gamma_0(t,s):=t(su_q+1-s)\in\mathcal C$ for every $(t,s)\in Q$. 
Furthermore, the energy levels are ordered as follows
\begin{equation}\label{eq:cf-energie}
I_q(u_q)  <I_q (1) < I_q(v_q) .
\end{equation}
The first inequality above is a consequence of the convergence in \eqref{eq:conv-uq}; the second one descends from \eqref{eq:ordine-energie} and implies that $v_q \not\equiv 1$ for $q$ sufficiently large. 

\subsection{Refined a priori estimates}
In the proof of Theorem \ref{thm:vqto1}, refined a priori estimates on $v_q$ and on its derivative are crucial.  Notice that, by standard elliptic regularity, every solution of the problem is $C^1(\bar B)$.

In the next lemma, thanks to a phase plane analysis, we refine, with respect to \cite[Lemma 2.2]{BFGM}, the $C^1$-a priori estimates for solutions of \eqref{eq:Pq} belonging to $\mathcal C$. 
Throughout the paper, for radial functions we use alternatively $u(x)$ and $\nabla u(x)$, with $x\in B$, or $u(r)$ and $u'(r)$, with $r=|x|\in(0,1)$, with abuse of notation. 

\begin{lemma}\label{commonbound}
Let $u\in \mathcal C $ be a solution of \eqref{eq:Pq}. For every $r\in [0,1]$ it holds 
\[
u(r) \leq \left(\frac{q}{p}\right)^\frac{1}{q-p} 
\quad\textrm{ and }\quad
u'(r) \leq \left(\frac{q-p}{q(p-1)}\right)^\frac{1}{p}.
\]
\end{lemma}
\begin{proof}
We follow the reasoning in the proof of \cite[Lemma 5.5]{BF}.
By integrating the equation satisfied by $u$, we get
\[
\int_B u^{p-1}(1-u^{q-p}) dx =0.
\]
If $u\equiv 1$, the statement is clearly verified. Otherwise, since $u\not\equiv1$ is positive and non-decreasing, we deduce that
\begin{equation}\label{eq:u(0)}
u(0)<1 \quad\textrm{ and }\quad u(1)>1 .
\end{equation}

Consider the equation satisfied by $u$ in radial form. We multiply it by $u'\ge 0$ to obtain 
\[
\left( \frac{p-1}{p} (u')^p +\frac{u^q}{q} - \frac{u^p}{p} \right)'
=-\frac{N-1}{r} (u')^p.
\]
We deduce that the function
\begin{equation}
L _q(r):=\frac{p-1}{p} (u'(r))^p - \frac{u(r)^p}{p} +\frac{u(r)^q}{q},\quad r\in[0,1]
\end{equation}
is non-increasing in $r$, and hence, using \eqref{eq:u(0)},
\[
L _q(r)\leq L _q(0) = -\frac{u(0)^p}{p}+\frac{u(0)^q}{q} \leq 0\quad\mbox{for all } r\in[0,1].
\]
We note that $L _q(r)\le 0$ is equivalent to 
\[
(u(r), u'(r))\in\Sigma:=\left\{(x,y)\in\mathbb R^2\,:\,x\ge0,\, 0\le y\le \left[\frac{p}{p-1}\left(\frac{x^p}p-\frac{x^q}q\right)\right]^{1/p} \right\},
\]
which implies the statement.
\end{proof}

As an immediate consequence of Lemma \ref{commonbound}, we infer the following $W^{1,p}$-estimate for solutions of \eqref{eq:Pq} belonging to $\mathcal{C}$:
\begin{equation}\label{eq:first_improvement}
\|u\|_{W^{1,p}(B)}^p \leq 
|B|\left(\frac{q-p}{q(p-1)}+\left(\frac{q}{p}\right)^{\frac{p}{q-p}}\right).
\end{equation}
Notice that the right hand side of the previous inequality converges to 
$|B| p/(p-1)$ as $q\to\infty$, so that such bound is uniform in $q$.
As this is too rough for our purposes, we shall now introduce an important tool that allows us to improve the $W^{1,p}$-estimate. 
It is an equivalent definition of weak solutions of \eqref{eq:Pq} belonging to $\mathcal{C}$; we believe that this characterization is interesting in itself. 

\begin{lemma}\label{lem:conseq-def-equiv}
A function $u\in \mathcal C$ is a weak solution of \eqref{eq:Pq} if and only if for every $\varphi\in W^{1,p}(B)$
\begin{equation}\label{eq:def-sol-equiv}
\int_B\frac{|\nabla u|^p+u^p}{p}\, dx\le \int_B\frac{|\nabla \varphi|^p+|\varphi|^p}{p}\,dx - \int_B u^{q-1}(\varphi-u)\, dx.  
\end{equation}
\end{lemma}
\begin{proof}
For $u\in \mathcal C$, let $\mathcal E_u : W^{1,p}(B) \to \mathbb R$ be defined as
\[
\mathcal E_u(\varphi):=\int_B\frac{|\nabla \varphi|^p+|\varphi|^p}{p} \, dx-\int_B u^{q-1}\varphi \,dx.
\]
Notice that a function $u\in \mathcal C$ is a weak solution of \eqref{eq:Pq} if and only if it is a critical point of $\mathcal E_u$. 
If $u$ satisfies \eqref{eq:def-sol-equiv}, $u$ is a global minimizer  and so a critical point of $\mathcal E_u$.
As for the other implication, we observe that $\mathcal E_u$ is the sum of $\frac{1}{p}\|\varphi\|^p_{W^{1,p}(B)}$ and a linear term, hence it is convex. Therefore, if $u$ is a critical point of $\mathcal E_u$, it is a global minimizer, and so \eqref{eq:def-sol-equiv} holds
for every $\varphi\in W^{1,p}(B)$, thus proving the statement.
\end{proof}

\begin{corollary}\label{cor:improved-est}
Let $u\in \mathcal{C}$ be a solution of \eqref{eq:Pq}, then
\[
I_q(u) \leq \frac{\|u\|_{W^{1,p}(B)}^p}{p} \leq 
\frac{|B|}{p} \left(\frac{q}{p}\right)^{\frac{p}{q-p}}.
\]
\end{corollary}
\begin{proof}
The first inequality is a consequence of the definition of $I_q$.
We choose $\varphi\equiv (q/p)^{\frac{1}{q-p}}$ in \eqref{eq:def-sol-equiv} to obtain
\[
\int_B\frac{|\nabla u|^p+u^p}{p}\, dx \le \left(\frac{q}{p}\right)^{\frac{p}{q-p}}\frac{|B|}{p} - \int_B u^{q-1}\left[\left(\frac{q}{p}\right)^{\frac{1}{q-p}}-u\right]\, dx
\]
and then use the first estimate in Lemma \ref{commonbound} to get the second inequality in the statement.
\end{proof}

\section{Asymptotics of $v_q$ as $q\to\infty$}\label{Sec:main}
In this section we shall prove Theorem \ref{thm:vqto1}.  
Let $p\in (1,2)$ and $v_q\in\mathcal C$ be a solution of  \eqref{eq:Pq} satisfying \eqref{eq:higher_energy}.
To start with, using the a priori bounds proved in the previous section, we show the existence of a limit profile $v_\infty$ of $(v_q)$ as $q\to\infty$, up to subsequences.

\begin{corollary}\label{cor:weakconv}
Let $q_n\to\infty$. There exist a subsequence, still denoted by $(q_n)$, and a function $v_\infty\in\mathcal C$ such that 
\begin{equation}\label{eq:convergences}
v_{q_n}\rightharpoonup v_\infty\;\mbox{ weakly in }W^{1,p}(B), \quad v_{q_n}\to v_\infty \;\mbox{ in }C^{0,\nu}(\bar B),
\end{equation}
as $n\to\infty$, for any $\nu\in(0,1)$.
\end{corollary}
\begin{proof} By Lemma \ref{commonbound}, and using that $(q/p)^{\frac{1}{q-p}}\to 1$ and $[(q-p)/(q(p-1))]^{\frac{1}{p}}\to (1/(p-1))^{\frac{1}{p}}$ as $q\to\infty$, $(v_{q_n})$ is bounded in the $C^1$-norm.
Using the compactness of the embedding $C^1\hookrightarrow C^{0,\nu}$, we infer the existence of a subsequence $(n_k)$ and a function $v_\infty$ for which $v_{q_{n_k}}\to v_\infty$ in $C^{0,\nu}(\bar B)$ for every $\nu\in (0,1)$. In particular, as $(v_{q_{n_k}})$ is bounded in $W^{1,p}(B)$, by possibly passing to a further subsequence, $v_{q_{n_k}}\rightharpoonup v_\infty$ weakly in $W^{1,p}(B)$. Furthermore, since $v_{q_{n_k}}\to v_\infty$ pointwise, $v_\infty$ is radial, non-negative and non-decreasing, i.e. $v_\infty\in\mathcal C$. 
\end{proof}

\begin{corollary}\label{weakconv}
Let $q_n\to\infty$ and let $v_\infty \in \mathcal C$ such that \eqref{eq:convergences} holds.
Then $v_\infty(1)=1$, and
\begin{equation}\label{eq:int-q-to0}
\lim_{n\to\infty}\int_B\frac{v_{q_n}^{q_n}}{q_n}\, dx =0.
\end{equation}  
\end{corollary}
\begin{proof} We follow the lines of \cite[Lemma 4.4]{BFGM}, see also {\cite[Lemma 5.6]{BF}}. We integrate the equation satisfied by $v_{q}$ to get 
\[
\int_B v_q^{p-1}(1-v_q^{q-p}) dx =0.
\]
Since $v_q\not\equiv 1$ by \eqref{eq:higher_energy}, and is positive and non-decreasing, we deduce that \eqref{eq:u(0)} holds for $v_q$.
Hence, $\|v_q\|_{L^{\infty}(B)}=v_q(1)>1$. Consequently, using the $L^\infty$-estimate given in Lemma \ref{commonbound}, we get 
\[1\le \lim_{n\to\infty}v_{q_n}(1)= \lim_{n\to\infty}\|v_{q_n}\|_{L^\infty(B)}\le \lim_{n\to\infty} \left(\frac{q_n}{p}\right)^{\frac{1}{q_n-p}}= 1,\]
and so $v_\infty(1)=\lim_{n\to\infty}v_{q_n}(1)=1$.

Finally, using the equation satisfied by $v_q$ and Corollary \ref{cor:improved-est}, we have
\[
\int_B \frac{v_q^q}{q}\, dx = \frac{\|v_q\|_{W^{1,p}(B)}^p}{q}
\le \frac{|B|}{q}\left(\frac{q}{p}\right)^\frac{p}{q-p},
\]
which implies \eqref{eq:int-q-to0}.
\end{proof}

%

So far we have proved that every sequence $(v_{q_n})$ has, up to subsequences,  a limit profile $v_\infty$. 
Our goal is to show that any limit profile $v_\infty$ is identically equal to $1$, which will also imply that the whole family $v_q$ converges. 
To this aim, we first detect the limit of the energy levels $I_q(v_q)$.
In view of \eqref{eq:int-q-to0}, the expected energy functional in the limit is simply $\frac{1}{p}\|\cdot\|^p_{W^{1,p}(B)}$, where the term involving the $q$-power disappears. In this sense, the following proposition states that $I_q(v_q)$ converges to the energy of the constant function $1$. 

\begin{proposition}\label{prop:conv-en}
The energy of $v_q$ has the following limit
\begin{equation}\label{eq:Iqvq}
\lim_{q\to\infty} I_q(v_q) =\frac{|B|}{p}=\frac{\|1\|^p_{W^{1,p}(B)}}{p}.
\end{equation}
%
\end{proposition}
\begin{proof}
By assumption \eqref{eq:higher_energy}, $I_q(v_q)>I_q(1)$ for every $q$.  Hence, we have
\begin{equation}\label{eq:liminf}
\liminf_{q\to\infty}I_q(v_q)\ge \lim_{q\to\infty}\left(\frac{1}{p}-\frac{1}{q}\right)|B|=\frac{|B|}{p}.
\end{equation}
In order to prove the reverse inequality for the supremum limit, we shall use in a crucial way the characterization of weak solutions of \eqref{eq:Pq} in the cone given in Lemma \ref{lem:conseq-def-equiv} and the consequent improved bound obtained in Corollary \ref{cor:improved-est}. Indeed, 
\[
\limsup_{q\to\infty} I_q(v_q) 
\leq \frac{|B|}{p}\lim_{q\to\infty}\left(\frac{q}{p}\right)^{\frac{p}{q-p}}
= \frac{|B|}{p},
\]
which, together with \eqref{eq:liminf}, proves the limit in \eqref{eq:Iqvq}.
%
\end{proof}

The convergence of the energy levels is not enough to conclude that $v_\infty\equiv 1$, the difficulty being that the convergence of $(v_q)$ to $v_\infty$, up to now, is not strong in the $W^{1,p}$-norm. We shall proceed by contradiction. To this aim, in the next lemma, we prove that if $v_\infty\not\equiv 1$, the convergence of $(v_q)$ to $v_\infty$ would be stronger in the region where $v_\infty<1$.  

\begin{lemma}\label{lem:conv-v'q}
Let $q_n\to\infty$ and let $v_\infty \in \mathcal C$ such that \eqref{eq:convergences} holds.
If there exists $R\in (0,1)$ such that $v_\infty<1$ in $[0,R]$, then $v'_{q_n}\to v'_\infty$ uniformly in $(0,R]$ as $n\to\infty$.
\end{lemma}
\begin{proof}
Exploiting the radial symmetry of the problem and that of $v_q$, we can write the equation for $v_q$ in radial form, to get 
\[
-(r^{N-1}(v'_q)^{p-1})'=r^{N-1}(v_q^{q-1}-v_q^{p-1}).
\]
We introduce $w_q:=r^{N-1}(v'_q)^{p-1}$ and observe that in the interval $[0,R]$ the following estimates hold for $q$ large and for a suitable $C>0$ independent of $q$: 
\[
\begin{gathered}
w_q\le \left(\frac{q-p}{q(p-1)}\right)^{\frac{p-1}{p}}\le C\\
|w'_q|\le r^{N-1}(v_q^{q-1}+v_q^{p-1})\le v_q^{q-1}(R)+v_q^{p-1}(R),
\end{gathered}
\]
where in the first line we have used the estimate of $v'_q$ given in Lemma \ref{commonbound}, while in the second line we have used that $v_q$ is non-decreasing.
Now, along the sequence $(v_{q_n})$, using the convergence $v_{q_n}(R)\to v_\infty(R)< 1$, we infer
\[
|w'_{q_n}|\le 2.
\]
Therefore, by the Arzel\`a-Ascoli Theorem, $w_{q_n}\to w_\infty$ uniformly in $[0,R]$, for a suitable function $w_\infty$. 

Now, for every $\delta>0$, we have in $[\delta, R]$
\[
\left|(v'_{q_n})^{p-1}-\frac{w_\infty}{r^{N-1}}\right|=\left|\frac{r^{N-1}(v'_{q_n})^{p-1}- w_\infty}{r^{N-1}}\right|\le \frac{\|w_{q_n}-w_\infty\|_{L^\infty(0,R)}}{\delta^{N-1}}\to 0
\]
as $n\to\infty$, whence
\[
\left|v'_{q_n}-\left(\frac{w_\infty}{r^{N-1}}\right)^{\frac{1}{p-1}}\right|\le\frac{1}{p-1}\left|(v'_{q_n})^{2-p}+\left(\frac{w_\infty}{r^{N-1}}\right)^{\frac{2-p}{p-1}}\right|\left|(v'_{q_n})^{p-1}-\frac{w_\infty}{r^{N-1}}\right|\to 0,
\]
where we have used the following algebraic inequality 
\[
|a^\gamma-b^\gamma|\le \gamma(|a|^{\gamma-1}+|b|^{\gamma-1})|a-b|\quad\mbox{for all }a,\,b\in \mathbb R,\,\gamma\ge 1,
\]
with $\gamma=\frac{1}{p-1}$, $a=(v'_{q_n})^{p-1}$ and $b=\frac{w_\infty}{r^{N-1}}$.
Therefore, for every small $\delta>0$, $(v'_{q_n})$ converges uniformly to $\left(\frac{w_\infty}{r^{N-1}}\right)^{\frac{1}{p-1}}$ in $[\delta,R]$. Since also $v_{q_n}\to v_\infty$ uniformly in $[\delta,R]$, by application of the uniform convergence to differentiability (see for instance \cite{Rudin}), $v_\infty$ is differentiable and $\left(\frac{w_\infty}{r^{N-1}}\right)^{\frac{1}{p-1}}=v'_\infty$ in $[\delta,R]$. By the arbitrariness of $\delta>0$, we get that $v'_{q_n}\to v'_\infty$ uniformly in $(0,R]$.
\end{proof}

\begin{proof}[$\bullet$ Proof of Theorem \ref{thm:vqto1}]
By Corollary \ref{cor:weakconv}, given any $q_n\to\infty$ there exist a subsequence, still denoted by $(q_n)$, and a function $v_\infty\in\mathcal C$ such that $v_{q_n}\rightharpoonup v_\infty$  weakly in $W^{1,p}(B)$ and $v_{q_n}\to v_\infty$ in $C^{0,\nu}(\bar B)$ for any $\nu\in(0,1)$.
We shall now prove that any limit function $v_\infty$ is the constant 1 (hence also implying the convergence along the whole family $v_q$). To do so, we suppose by contradiction that $v_\infty\not\equiv 1$ and define
\[
\bar{R}:=\inf\{r\in[0,1]\,:\, v_\infty(r)=1\}.
\]
Since $v_\infty(1)=1$, $\bar{R}\le 1$, and by the contradiction assumption, $\bar{R}>0$. Furthermore, by the definition of $\bar{R}$ as infimum, $v_\infty<1$ in $[0,\bar{R})$, and using that $v_\infty$ is continuous, non-decreasing, and that $v_\infty(1)=1$, it holds that $v_\infty\equiv1$ in $[\bar{R},1]$.

We divide the proof into two main steps. 
\smallskip

{\sc Step 1.} In this first step we prove that $v_\infty|_{B_{\bar{R}}}$ solves the problem 
\begin{equation}\label{eq:limit-pb}
\begin{cases}
-\Delta_p u+|u|^{p-2}u=0\quad&\mbox{in }B_{\bar{R}},\\
u=1&\mbox{on }\partial B_{\bar{R}}
\end{cases}
\end{equation}
and that 
\begin{equation}\label{eq:conseq-step1}
\|v_\infty\|^p_{W^{1,p}(B)}<|B|.
\end{equation} 
To prove \eqref{eq:limit-pb}, we aim to show that the distributional identity 
\[
\int_{B_{\bar{R}}}|\nabla v_\infty|^{p-2}\nabla v_\infty\cdot\nabla \varphi\,dx+\int_{B_{\bar{R}}} v_\infty^{p-1} \varphi\,dx=0
\]
holds for every $\varphi\in C_c^\infty(B_{\bar{R}})$. We observe that, fixed $\varphi\in C_c^\infty(B_{\bar{R}})\subset C^\infty(\bar B)$, since $v_q$ solves \eqref{eq:Pq}, 
\begin{equation}\label{eq:distr-id-vq}
\int_{B_{\bar{R}}}|\nabla v_q|^{p-2}\nabla v_q\cdot\nabla \varphi\,dx+\int_{B_{\bar{R}}} v_q^{p-1} \varphi\,dx=\int_{B_{\bar{R}}} v_q^{q-1} \varphi\,dx.
\end{equation}
Now, using the estimate in Lemma \ref{commonbound}, for $q$ large, $|v_q^{p-1}\varphi|\le 2|\varphi|\in L^1(B_{\bar{R}})$, moreover $v_{q_n}^{p-1}\varphi\to v_\infty^{p-1}\varphi$ pointwise in $B_{\bar{R}}$. Thus, by the Dominated Convergence Theorem
\begin{equation}\label{eq:conv1}
\int_{B_{\bar{R}}}v_{q_n}^{p-1}\varphi\,dx \to \int_{B_{{\bar{R}}}}v_\infty^{p-1}\varphi\,dx \quad\text{as }n\to\infty.
\end{equation}
Let $R_\varphi<{\bar{R}}$ be such that $\mathrm{supp}\varphi\subseteq B_{R_\varphi}$. By the definition of ${\bar{R}}$, $v_\infty(R_\varphi)<1$ and so, in correspondence of $\varepsilon=\frac{1-v_\infty(R_\varphi)}{2}>0$, there exists $\bar n$ large such that 
\[
v_{q_n}(r)\le v_{q_n}(R_\varphi)\le v_\infty(R_\varphi)+\varepsilon=\frac{v_\infty(R_\varphi)+1}{2}<1\quad\mbox{for every }n\ge \bar{n}
\]
for every $r\le R_\varphi$.
This implies that, as $n\to\infty$,
\begin{equation}\label{eq:conv2}
\int_{B_{\bar{R}}}v_{q_n}^{{q_n}-1}\varphi\,dx= 
\int_{B_{R_\varphi}}v_{q_n}^{{q_n}-1}\varphi\,dx\le 
\left(\frac{v_\infty(R_\varphi)+1}{2}\right)^{q-1}\int_{B_{R_\varphi}}|\varphi|\,dx\to 0.
\end{equation}
It remains to study the convergence of the first integral in \eqref{eq:distr-id-vq} involving the gradients. We observe that such integral is actually over $B_{R_\varphi}$, therefore we can estimate 
\[
\begin{aligned}
&\left|\int_{B_{R_\varphi}}|\nabla v_q|^{p-2}\nabla v_q\cdot\nabla \varphi\,dx-\int_{B_{R_\varphi}}|\nabla v_\infty|^{p-2}\nabla v_\infty\cdot\nabla \varphi\,dx\right|\\
&\le \left|\int_{B_{R_\varphi}}(|\nabla v_q|^{p-2}-|\nabla v_\infty|^{p-2})\nabla v_q\cdot\nabla \varphi\,dx\right|\\
&\phantom{\le}+\left|\int_{B_{R_\varphi}}|\nabla v_\infty|^{p-2}\nabla v_q\cdot\nabla \varphi\,dx-\int_{B_{R_\varphi}}|\nabla v_\infty|^{p-2}\nabla v_\infty\cdot\nabla \varphi\,dx\right|\\
&\le \|(v'_q)^{p-2}-(v'_\infty)^{p-2}\|_{L^\infty(0,R_\varphi)}\left(\int_{B_{R_\varphi}}C|\nabla \varphi|\,dx\right)+|F(v_q)-F(v_\infty)|,
\end{aligned}
\]
where in the last line we have used the Cauchy-Schwartz inequality, the constant $C>0$ independent of $q$ comes from the estimate on $v'_q$ given in Lemma \ref{commonbound}, and we have introduced the functional $F\in (W^{1,p}(B))'$ defined as $F(u)= \int_{B_{R_\varphi}}|\nabla v_\infty|^{p-2}\nabla u\cdot\nabla \varphi\,dx$ for every $u\in W^{1,p}(B)$. 
Applying Lemma \ref{lem:conv-v'q} with $R=R_\varphi<{\bar{R}}$, we deduce that $\|(v'_{q_n})^{p-2}-(v'_\infty)^{p-2}\|_{L^\infty(0,R_\varphi)}\to 0$ as $n\to\infty$. Moreover, by the weak convergence of $(v_{q_n})$ in $W^{1,p}(B)$, we have $|F(v_q)-F(v_\infty)|\to 0$. Therefore, the previous chain of inequalities provides
\[
\int_{B_{{\bar{R}}}}|\nabla v_{q_n}|^{p-2}\nabla v_{q_n}\cdot\nabla \varphi\,dx\to \int_{B_{{\bar{R}}}}|\nabla v_\infty|^{p-2}\nabla v_\infty\cdot\nabla \varphi\,dx\quad \mbox{as }n\to \infty.
\]
Combining this limit with \eqref{eq:conv1} and \eqref{eq:conv2}, we get that $v_\infty|_{B_{\bar{R}}}$ solves \eqref{eq:limit-pb}. 

Now, we observe that, by \cite[Lemma 5.7]{BF} (see also \cite[Lemma 4.5]{BFGM}), the following infimum 
\[
\inf \left\{ \|u\|^p_{W^{1,p}(B_{\bar{R}})}\,:\, u\in \mathcal C,\, u\equiv 1\mbox{ on }\partial B_{\bar{R}} \right\}
\]
is uniquely achieved by the solution of \eqref{eq:limit-pb}, i.e. by $v_\infty|_{B_{\bar{R}}}$. Therefore we obtain that 
\[
\|v_\infty\|^p_{W^{1,p}(B)}=\|v_\infty\|^p_{W^{1,p}(B_{\bar{R}})}+\|1\|^p_{W^{1,p}(B\setminus B_{\bar{R}})}<\|1\|^p_{W^{1,p}(B_{\bar{R}})}+\|1\|^p_{W^{1,p}(B\setminus B_{\bar{R}})}=|B|.
\]
\smallskip

{\sc Step 2.} In this second step we show that $I_{q_n}(v_{q_n})\to \frac{1}{p}\|v_\infty\|^p_{W^{1,p}(B)}$, more explicitly
\[
\frac{1}{p}\int_B|\nabla v_{q_n}|^p\,dx+\frac{1}{p}\int_B v_{q_n}^p\,dx-\frac{1}{{q_n}}\int_B v_{q_n}^{q_n}\,dx\to \frac{1}{p}\int_B|\nabla v_\infty|^p\,dx+\frac{1}{p}\int_B v_\infty^p\,dx
\]
as $n\to \infty$.
We already know by \eqref{eq:int-q-to0} that $\frac1{q_n}\int_B v_{q_n}^{q_n}\,dx\to 0$. Moreover, since $W^{1,p}(B)$ is compactly embedded in $L^p(B)$ and $v_{q_n}\rightharpoonup v_\infty$ in $W^{1,p}(B)$, $\int_B v_{q_n}^p \,dx\to \int_B v_\infty^p \,dx$. It remains to study the convergence of the gradient term. Now, for every small $\delta>0$ we can write
\[
\int_B|\nabla v_{q_n}|^p\,dx=\begin{cases}
\int_{B_{{\bar{R}}-\delta}}\dots dx+\int_{B_{{\bar{R}}+\delta}\setminus B_{{\bar{R}}-\delta}}\dots dx+\int_{B\setminus B_{{\bar{R}}+\delta}}\dots dx\;&\mbox{if }{\bar{R}}<1,\\
\int_{B_{1-\delta}}\dots dx+\int_{B\setminus B_{1-\delta}}\dots dx &\mbox{if }{\bar{R}}=1.
\end{cases}
\]
In both cases ${\bar{R}}<1$ and ${\bar{R}}=1$, the first integral $\int_{B_{{\bar{R}}-\delta}}|\nabla v_{q_n}|^p\,dx$ can be treated in the same way:
\begin{equation}\label{eq:first-int}
\begin{aligned}
&\left|\int_{B_{{\bar{R}}-\delta}}|\nabla v_{q_n}|^p\,dx-\int_{B_{{\bar{R}}-\delta}}|\nabla v_\infty|^p\,dx\right|\le \int_{B_{{\bar{R}}-\delta}}\left||\nabla v_{q_n}|^p-|\nabla v_\infty|^p\right|\,dx\\
&\hspace{3cm}\le \|(v'_{q_n})^p-(v'_\infty)^p\|_{L^\infty(0,{\bar{R}}-\delta)}|B_{{\bar{R}}-\delta}|\to 0 \quad\mbox{as }n\to \infty,
\end{aligned}
\end{equation} 
where we have used Lemma \ref{lem:conv-v'q} with $R=\bar{R}-\delta$. Using the same argument it is possible to treat also the second integral in both cases $\int_{B_{{\bar{R}}+\delta}\setminus B_{{\bar{R}}-\delta}}|\nabla v_{q_n}|^p\, dx$ and $\int_{B\setminus B_{1-\delta}}|\nabla v_{q_n}|^p\, dx$. Namely, if ${\bar{R}}<1$
\begin{equation}\label{eq:second-int}
\int_{B_{{\bar{R}}+\delta}\setminus B_{{\bar{R}}-\delta}}|\nabla v_{q_n}|^p\, dx\le \frac{{q_n}-p}{{q_n}(p-1)}|B_{{\bar{R}}+\delta}\setminus B_{{\bar{R}}-\delta}|\le C|B_{{\bar{R}}+\delta}\setminus B_{{\bar{R}}-\delta}|, 
\end{equation}
where $C>0$ is a suitable constant not dependent on $n$ nor on $\delta$, and we have used the estimate for $v'_{q}$ in Lemma \ref{commonbound}. A similar estimate holds for the corresponding integral when ${\bar{R}}=1$. 
Finally, in the case ${\bar{R}}<1$, it remains to consider the third integral $\int_{B\setminus B_{{\bar{R}}+\delta}}|\nabla v_{q_n}|^p\, dx$. Since $v_\infty\equiv 1$ in $[{\bar{R}}+\delta,1]$, we need to prove that $\int_{B\setminus B_{{\bar{R}}+\delta}}|\nabla v_{q_n}|^p\, dx\to 0$ as $q\to \infty$. It suffices to prove that $v'_{q_n}\to 0$ in $[{\bar{R}}+\delta,1]$. 
We argue by contradiction and suppose that there exist $\bar{r}\in[{\bar{R}}+\delta,1)$, a subsequence $(q_{n_k})$, and a positive number $a>0$ such that $v'_{q_{n_k}}(\bar{r})\ge a$ for every $n\in\mathbb N$. Then, for every $r\in (\bar{r},1]$,
\begin{equation}\label{eq:fund-calc}
v_{q_{n_k}}(r)-v_{q_{n_k}}(\bar{r})=\int_{\bar{r}}^r v'_{q_{n_k}}(\tau) \,d\tau\ge a(r-\bar{r})>0. 
\end{equation}
On the other hand, $v_{q_n}\to 1$ uniformly in $[{\bar{R}}+\delta,1]$, and so the left-hand side of \eqref{eq:fund-calc} tends to zero as $q_{n_k}\to \infty$. This is a contradiction. 

By the arbitrariness of $\delta>0$, passing to the limit as $\delta\to 0$ in \eqref{eq:second-int}, we conclude the proof of Step 2. 
\smallskip

{\sc Conclusion.} Combining the result proved in Step 2. with \eqref{eq:Iqvq}, we get that 
\begin{equation}\label{eq:norma-v-infty}
\|v_\infty\|^p_{W^{1,p}(B)}=|B|. 
\end{equation}
This contradicts \eqref{eq:conseq-step1}, thus proving that $v_\infty\equiv 1$. Moreover, in view of Corollary \ref{cor:weakconv}, $v_{q_n}$ converges to $1$ in $C^{0,\nu}(\bar B)$ and weakly in $W^{1,p}(B)$. By \eqref{eq:int-q-to0}, Proposition \ref{prop:conv-en} and relation \eqref{eq:norma-v-infty}
\[
\lim_{n\to\infty}\frac{\|v_{q_n}\|^p_{W^{1,p}(B)}}{p}=\lim_{n\to\infty}I_{q_n}(v_{q_n})=\frac{|B|}{p}=\frac{\|v_\infty\|^p_{W^{1,p}(B)}}{p}.
\]
Therefore, $\|v_{q_n}\|_{W^{1,p}(B)}\to \|v_\infty\|_{W^{1,p}(B)}$, which together with the weak convergence implies that $v_{q_n}\to v_\infty\equiv 1$ in $W^{1,p}(B)$. By the arbitrariness of the sequence, the convergence of the whole family follows.
\end{proof}

\section*{Acknowledgments}
The authors were partially supported by the INdAM-GNAMPA Project 2022 ``Studi asintotici in problemi parabolici ed ellittici ''.

The authors wish to thank the anonymous referees for their useful suggestions.

\noindent

\bibliographystyle{abbrv}

\end{document}